\documentclass[12pt]{amsart}
\usepackage[]{amsmath}
\usepackage{slashed}
\usepackage{amssymb}
\usepackage[colorlinks]{hyperref}
\usepackage{amsthm}
\usepackage[]{geometry}
\usepackage{tikz}
\usetikzlibrary{matrix}
\usetikzlibrary{arrows}

\newcommand{\mb}{\mathbb}

\newcommand{\mc}{\mathcal}
\newcommand{\mt}{\text}
\newcommand{\rr}{\mb R}
\newcommand{\cc}{\mb C}
\newcommand{\oo}{\mb O}
\newcommand{\ff}{\mb F}
\newcommand{\ee}{\mb E}
\newcommand{\vv}{\mb V}

\newcommand{\id}{\mt{I}}
\newcommand{\re}[1]{\mt{Re}(#1)}
\newcommand{\im}[1]{\mt{Im}(#1)}
\newcommand{\reo}{\re\oo}
\newcommand{\imo}{\im\oo}
\newcommand{\conj}{\overline}
\newcommand{\dirac}{\slashed{D}}
\newcommand{\Di}{\dirac}

\newcommand{\Hom}[3]{\mt{Hom}_{#1}({#2},{#3})}

\newcommand{\GL}{\mt{GL}}
\newcommand{\SL}{\mt{SL}}

\newcommand{\U}{\mt{U}}
\newcommand{\Sp}{\mt{Sp}}
\newcommand{\cg}{G_2^\cc}

\newcommand{\pb}{^*}
\newcommand{\dual}{^*}
\newcommand{\inv}{^{-1}}
\newcommand{\vol}{\mt{Vol}}
\newcommand{\T}{\mathsf{T}}
\newcommand{\wt}{\widetilde}

\newtheorem{theorem}{Theorem}
\newtheorem{proposition}{Proposition}
\newtheorem{definition}{Definition}
\newtheorem{remark}{Remark}
\newtheorem{example}{Example}

\newtheorem{lemma}{Lemma}

\begin{document}
\title[Complex $G_2$-manifolds and Seiberg-Witten Equations]{Complex $G_2$-manifolds and \\ Seiberg-Witten Equations}
\author{Selman Akbulut and \"Ust\"un Y\i ld\i r\i m}
\thanks{First named author is partially supported by NSF grant 1505364}
\keywords{$G_2$ manifold, complex $G_2$ manifold, associative submanifold} 
\address{Department  of Mathematics, Michigan State University, East Lansing, MI, 48824}
\email{akbulut@msu.edu }
\email{ustun@mailbox.org}
\subjclass[2010]{53C38, 53C29, 57R57}

\begin{abstract}
  We introduce the notion of complex $G_2$ manifold $M_{\cc}$, and complexification of a $G_2$ manifold $M\subset M_{\cc}$. As an application we show the following: If $(Y,s)$ is a closed oriented $3$-manifold with a $Spin^{c}$ structure, and $(Y,s)\subset (M, \varphi)$ is an imbedding as an associative submanifold of some $G_2$ manifold (such imbedding always exists), then the isotropic associative deformations of $Y$ in the complexified $G_2$ manifold $M_{\cc}$  is given by Seiberg-Witten equations. 

\end{abstract}
\maketitle

\vspace{-0.5in}

\tableofcontents

\section{Introduction}
An {\it almost $G_{2}$ manifold} $(M^7, \varphi)$  is a $7$-manifold whose tangent frame bundle reduces to the Lie group $G_{2}$. Sometimes $G_2$ manifolds are called  {\it manifolds with $G_2$ structure}. This structure is determined by a certain ``{\it positive}'' $3$-form $\varphi$, which in turn  induces a metric $g$ and a cross product $\times $ structure on the tangent bundle $TM$. An almost $G_2$ manifold is called a  $G_2$ manifold, if the induced metric has its holonomy group contained in $G_2$ (e.g. \cite{Bry05}). There is an interesting class of submanifolds of $Y^{3}\subset (M,\varphi)$ called associative submanifolds, they are the submanifolds where $\varphi$ restricts to the volume form of $Y$, equivalenty the tangent space $TY$ is closed under the cross product operation.

\vspace{.05in}

By \cite{RS07} every oriented 3-manifold $Y$ embeds into some $G_2$ manifold $(M,\varphi)$ as an associative submanifold. In fact, any oriented $3$-manifolds with a $Spin^{c}$ structure $(Y,s)$ embeds into a $G_2$ manifold with a $2$-frame field $(M, \varphi, \Lambda)$ as an associative submanifold, such that  $s$  is induced from the $2$-frame field $\Lambda $ \cite{AS08}.

\vspace{.05in}

In \cite{M98} McLean showed that the local deformations of the associative submanifold of a $G_2$ manifold can be identified with the kernel of a certain Dirac operator 
\linebreak $ \Di_{\bf A_{0}}: \Omega^{0}(\nu_Y) \to \Omega^{0}(\nu_Y)$, which is defined on the sections of the complexified normal bundle $\nu_{Y}$ of $Y$. In \cite{AS08} this result was extended to almost $G_2$ manifolds, by expressing the Dirac operator in terms of the cross product operation, and deforming its connection term $A_{0}\to A=A_{0} +a $,  by a $1$-form parameter $a\in \Omega^{1}(Y)$. This parameter makes the Dirac operator unobstructed:

 \begin{equation}
 \Di_{A_{0} +a} (v)= \sum e^{j}\times \nabla_{e_{j}}(v) + a(v).
 \end{equation}

\noindent By coupling this with a second equation we get Seiberg-Witten equations on $Y$: 

  \begin{equation}
    \label{eq:sweqs}
\begin{array}{l}
  \Di_{\bf A}(x) =0 \\
*F_{A}=   \sigma(x).
\end{array}
  \end{equation}
  
\noindent 
The second term can be written as $da=q(x)$ where $q(x)$ is some quadratic function.
This relates the Seiberg-Witten equations to the local deformation equations of the associative submanifolds, but they are not equivalent.  For one thing, these equations take place in the spinor bundle of $\nu_{Y}$ (complexification of $\nu_{Y}$) not in $\nu_{Y}$, and the misterious second equation of (2) has no apparent relation to the deformation of the associative submanifolds. Our motivation in writing this paper was to seek a larger manifold containing $Y$ with more structure, so that deformation equations of $Y$ in that manifold would be equivalent to both of the Seiberg-Witten equations, giving us a completely natural way to derive Seiberg-Witten equations from associative deformations. Here we achieve this goal by defining the notion of {\it complex $G_2$ manifold} $M_{\cc}$, and the notion of  {\it complexification} of a $G_2$ manifold $M\subset M_{\cc} \cong T^{*}M$.

\vspace{.05in}

\begin{theorem}
  Let  $(Y,s)$ be a  closed oriented  $3$-manifold with a $Spin^{c}$ structure, and $(Y,s)\subset (M, \varphi)$ be an imbedding as an associative submanifold of some $G_2$ manifold (note that such imbedding always exists). Then the  isotropic associative deformations of $(Y,s)$ in the complexified $G_2$ manifold $M_{\cc}$  is given by Seiberg-Witten equations (\ref{eq:sweqs}). 
\end{theorem}

\vspace{.05in}

A $\cg$ manifold  $M_{\cc}$ is a  $14$-dimensional almost complex manifold $(M,J)$ whose tangent frame bundle reduce to the complex Lie group $\cg$,  and the complexification of a $G_2$ manifold $M\subset M_{\cc}$ is just the inclusion to the cotangent bundle $M\subset T^{*}M$ as the zero section. This endows $M_{\cc}$ with richer structures then $M$, namely a complex $3$-form, a symplectic form, and a positive definite metric $\{ \varphi_\cc, \omega, g \}$. Then if we start with an associative submanifold $Y\subset M$ of a $G_2$ manifold $M$, and complexify $M \subset M_{\cc}$, and then deform $Y$ inside of $M_{\cc}$ as an {\it isotropic} associative submanifold of $M_{\cc}$,  amazingly we get the Seiberg-Witten type equations we are looking for (Dirac equation plus a quadratic equation). 

\vspace{.1in}

To keep our notations consistent, we wrote our constructions from ground up starting with the relevant vector spaces. First we discuss $G_{2}$ and $\cg$ vector spaces and the various forms on them, and study some compatible structures. Then we study various Grassmann manifolds and their relation to each other. The Grassmannians $G^{\varphi}_{3}(\rr^7)\subset G_{3}(\rr^7)$ are studied in \cite{AK16}, and  $G^{\varphi}_{3}(\cc^7)$ along with its smooth compactification  in $G_{3}(\cc^7)$ is studied in \cite{AC15}. Then we discuss complexification of a $G_2$ manifold, and in Section 6 we prove our deformation result.

\vspace{.1in}

Let us remark on interesting parameter of $G_2$ manifolds introduced in \cite{AS08}, which has some relevance here: Given a $G_2$ manifold $(M,\varphi)$ we can always choose a non-vanishing $2$-frame field $\Lambda=\{u,v\}$ (this exists on any spin $7$-manifold by \cite{T67}), then $\{u,v, u\times v\}$ generates a non-vanishing $3$-frame field on $TM$, then by using the induced metric we get the decomposition $TM=E\oplus V$, where $V=E^{\perp}$. Furthermore any unit section $\xi\in\Gamma(E)$ (there are 3 independent ones) gives a complex structure on $V$ by cross product map $J_{\xi}: V\to V$.   In particular, this says that the tangent bundle $TM$ of any $G_2$ manifold reduces to an $SU(2)$ bundle, as the $3$-dimensional trivial bundle plus $4$-dimensional HyperKahler bundle. Also by using this $\Lambda$,  the $spin^{c}$ structures on moving associative submanifolds $Y\subset M$ (which is used to define Seiberg-Witten equation), can be made to be induced from the global parameter $(M, \Lambda)$. We will address integrability conditions and analysis of the quadratic term in a future paper.

\section{Linear algebra}
Let $V_i$ be a vector space (over $\rr$) of dimension $2n_i$ with almost complex structures $J_i$ for $i=1,2$.
As usual, we can view $V_i$ as a complex vector space by setting $(a+ib)v=av+bJ_iv$.
(Here, we are using $i$ to denote both the index and $\sqrt{-1}$ but which one we mean will always be clear in a given context.)
Let $\left\{ e_i^j \right\}$ be a complex basis for $V_i$ and set $f_i^j=J_ie_i^j$.
Then, with respect to the real bases $\left\{ e_i^j,f_i^j \right\}$,
\begin{equation*}
  J_j = \begin{pmatrix}
    0 & -\id_{n_j} \\
    \id_{n_j} & 0
  \end{pmatrix}
\end{equation*}
and the embedding $\cc^{n_2\times n_1}\to \rr^{2n_2\times 2n_1}$ induced from $\Hom\cc{V_1}{V_2}\subset \Hom\rr {V_1}{V_2}$ is given by
\begin{equation*}
  X+iY \mapsto
  \iota(X+iY) =
  \begin{pmatrix}
    X & -Y \\
    Y & X
  \end{pmatrix}
\end{equation*}
for $X,Y\in\rr^{n_2\times n_1}$.
It is also easy to see that $M\in \rr^{2n\times 2n}$ is $J$-linear if and only if it is in this form.
In particular, we can identify $A=X+iY\in \mt{GL}(n,\cc)$ with its image $\iota(A)\in \mt{GL}(2n,\rr)$.

\subsection{Symmetric bilinear forms and \texorpdfstring{$S^1$}{circle} family of metrics}
By a metric $g$ on a vector space $V$, we mean a non-degenerate, $\rr$-bilinear map $g:V\times V\to \rr$.

Let $(V,B)$ be a vector space over $\cc$ with a symmetric (non-degenerate) bilinear form.
We define the orthogonal group O($V,B$) to be the subgroup of GL$(V)$ that preserves $B$.
More precisely,
\begin{equation}
  \mt O(V,B) = \left\{ A\in\mt{GL}(V)\;|\; B(Au,Av)=B(u,v) \;\;\mt{for all } u,v\in V \right\}.
\end{equation}

One can always find an orthonormal basis $\left\{ e_i \right\}$ for $B$ such that $B(e_i,e_j)=\delta^i_j$.
Matrix representation of an orthogonal transformation $A\in\mt O(V,B)$ satisfies the usual identity $A^\T A=\id$.
Decomposing $A=X+iY$ into real and complex parts, we get 
\begin{eqnarray*}
  X^\T X - Y^\T Y &=&  \id \\
  X^\T Y + Y^\T X &=& 0.
\end{eqnarray*}
Using these relations, it is easy to see that
\begin{equation*}
  \iota(A)^\T
  \begin{pmatrix}
    \id & 0 \\
    0 & -\id 
  \end{pmatrix}
  \iota(A) =
  \begin{pmatrix}
    \id & 0 \\
    0 & -\id 
  \end{pmatrix}.
\end{equation*}
In other words, it preserves the standard signature $(n,n)$ metric (with respect to the $\rr$-basis $\left\{ e_i, ie_i \right\}$).
More invariantly, this metric is given by $g = \re B$.
Note that $\im B = -\re {B(iu,v)}$.
So, fixing a complex structure on $V$, one can describe the imaginary part of $B$ in terms of its real part.
In fact, this leads to an $S^1$ family of $(n,n)$ metrics on $V$ by 
\begin{equation}
\re{B((\cos(t)-i\sin(t))x,y)}=\cos(t)\re B + \sin(t)\im B.
\label{eq:s1fammetrics}
\end{equation}
Indeed we are only applying an invertible linear map to the first variable of the LHS  of (\ref{eq:s1fammetrics}).
Therefore, since at $t=0$ it is a $(n,n)$ metric, it is a $(n,n)$ metric for all $t$.
\begin{remark}
Note that since $B$ is complex bilinear, its real part satisfies $g(iu,iv)=-g(u,v)$.
\label{rmk:anticompg}
\end{remark}
\begin{remark}
Also, note that the $S^1$ family of metrics connects $g$ to $-g$ so any such family has to consist of $(n,n)$ metrics. 
\end{remark}

Next, we consider the converse.
\begin{definition}
Let $(V,g,J)$ be a vector space with an inner product $g$, and an almost complex structure $J$.
We call $g$ and $J$ skew-compatible if $g(Ju,Jv)=-g(u,v)$ for all $u,v\in V$.
\label{defn:anticompt}
\end{definition}
\begin{remark}
Note that a skew-compatible $J$ is self adjoint, i.e.
\begin{equation}
g(Ju,v)=g(u,Jv).
\label{eq:slfadj}
\end{equation}
\end{remark}
Let $(V,J)$ be a $\rr$-vector space with an almost complex structure.
We view $V$ as a $\cc$-vector space by setting $(a+ib)v=av+bJv$ for $a,b\in\rr$, and $v \in V$.
\begin{proposition}
Let $g$ be an inner product on $V$ which is skew-compatible with $J$.
Then we can define a complex symmetric bilinear form $B$ on $V$ by 
$$B(u,v)=g(u,v)-ig(Ju,v).$$
\end{proposition}
\begin{proof}
By using (\ref{eq:slfadj}) and the fact that $g$ is symmetric, it is clear that $B$ is also symmetric.
The linearity over $\rr$ is also clear.
So, we only need to show $B(Ju,v)=iB(u,v)$.
\begin{align*}
B(Ju,v) &= g(Ju,v)-ig(J^2u,v) \\
&= g(Ju,v)+ig(u,v) \\
&= i\left( g(u,v)-ig(Ju,v) \right) \\
&= iB(u,v)
\end{align*}
\end{proof}
\begin{proposition}
Let $(V,g,J)$ be a vector space of dimension $2n$ with skew-compatible $g$ and $J$.
Then, $g$ is necessarily an $(n,n)$ metric on $V$.
\end{proposition}
\begin{proof}
Follows from the discussions above.
\end{proof}

Next, we take yet another step back.
Namely, we start with a $\rr$ vector space $V$ of dimension $n$ with a non-degenerate metric $g$.
We consider the complexifications $V_\cc$ and $g_\cc$ of $V$ and $g$ where
\begin{align*}
V_{\cc}&= V\otimes \cc = V\oplus iV \\
g_\cc(u+iv,u+iv) &= g(u,u) - g(v,v) + 2ig(u,v)
\end{align*}
for $u,v \in V$.
Clearly, $g_{\cc}$ is a symmetric bilinear form on $V_{\cc}$.
Hence, again by the above discussion, we get a $S^1$ family of $(n,n)$ metrics on $V_{\cc}$.
Later in subsection \ref{ssec:complexg2space}, we will take this construction one step further.

Another common construction along these lines is to complexify a vector space with an almost complex structure $(V,J)$.
We set $V_\cc=V\oplus iV$ as before and define $J_\cc(u+iv)=J(u)+iJ(v)$.
Further, we set
\begin{align*}
V^{1,0}&= \left\{ u\in V_{\cc} \; | \; J_\cc(u)=iu \right\} \text{ and} \\
V^{0,1}&= \left\{ u\in V_{\cc} \; | \; J_\cc(u)=-iu \right\}.
\end{align*}
This is the usual eigenspace decomposition for $J_\cc$.
We have the following projection maps $\xi:V_\cc\to V^{1,0}$ and $\conj\xi:V_\cc\to V^{0,1}$ defined by
\begin{equation*}
\xi(u)=\frac{1}{2}\left( u-iJ_{\cc}(u) \right).
\end{equation*}
Note that $\xi\big|_{V}$ is a $\cc$-linear isomorphism (with respect to $J$ on the domain).
If we have a symmetric $\cc$-bilinear form $B$ on $V$, then we can restrict to its real part $g=\re B$ and then complexify $g_\cc$.
Using $\xi$ we can compare these two symmetric bilinear forms on $V$ and on $V^{1,0}$.
\begin{proposition}
For $u,v\in V$,
\begin{equation*}
g_{\cc}(\xi u,\xi v) = \frac{1}{2}B(u,v).
\end{equation*}
\end{proposition}
\begin{proof}
  This is a straight forward computation as follows.
\begin{align*}
g_{\cc}(\xi u, \xi v) &= \frac{1}{4}g_{\cc}(u-iJ_{\cc}u,v-iJ_{\cc}v) \\
&= \frac{1}{4}\left( g(u,v)-g(Ju,Jv) -i\left( g(Ju,v)+g(u,Jv) \right) \right) \\
&= \frac{1}{4}\left( 2g(u,v)-2ig(Ju,v)\right) \\
&= \frac{1}{2}\left( g(u,v)-ig(Ju,v) \right) \\
&= \frac{1}{2}B(u,v)
\end{align*}
since by Remark \ref{rmk:anticompg}, $g$ and $J$ are skew-compatible and therefore $J$ is also self-adjoint.
\end{proof}

\subsection{The group \texorpdfstring{$\cg$}{complex G2}}
Let $(\oo,B)$ be an octonion algebra over $\cc$ (see \cite{SV13}).
We can associate a quadratic form $Q$ to $B$ in the standard way: $Q(u)=B(u,u)$.
In the other direction, we have
\begin{equation}
  B(u,v) = \frac{1}{2}\left( Q(u+v)-Q(u)-Q(v) \right).
  \label{eq:quad2bil}
\end{equation}

Octonions satisfy
\begin{equation}
Q(uv)=Q(u)Q(v) \qquad \text{for }u,v\in\oo.
  \label{eq:comp}
\end{equation}
\begin{proposition}
  \label{prop:ortmult}
  For $u,v,v'\in \oo$, we have $Q(u)B(v,v')=B(uv,uv')$ and $B(v,v')Q(u)=B(vu,v'u)$.
  In particular, for $Q(u)=1$, left or right multiplication by $u$ is an orthogonal transformation of $(\oo,B)$.
\end{proposition}
\begin{proof}
  \begin{eqnarray*}
    Q(u)B(v,v') &=& \frac{1}{2}Q(u)\left[ Q(v+v')-Q(v)-Q(v') \right ] \\
    &=& \frac{1}{2}\left[ Q(u)Q(v+v')-Q(u)Q(v)-Q(u)Q(v') \right ] \\
    &=& \frac{1}{2}\left[ Q(u(v+v'))-Q(uv)-Q(uv') \right ] \\
    &=& \frac{1}{2}\left[ Q(uv+uv')-Q(uv)-Q(uv') \right ] \\
    &=& B(uv,uv')
  \end{eqnarray*}
  The other equality can be proved similarly.
\end{proof}

\begin{definition}
We define $\cg$ to be the automorphism group of $\oo$.
\end{definition}
\begin{proposition}
  $\cg\le O(\oo,B)$
  \label{prop:ort}
\end{proposition}
\begin{proof}
  Let $A\in\cg$ and $u,v\in \oo$.
  \begin{eqnarray*}
    B(Au,Av) &=& \re{\conj{Au}Av} \\
    &=& \re{A\conj{u}Av} \\
    &=& \re{A\conj uAv} \\
    &=& \re{A(\conj uv)} \\
    &=& \re{\conj uv} \\
    &=& B(u,v)
  \end{eqnarray*}
\end{proof}
\begin{remark}
In fact, in \cite{SV13}, it is proved that $\cg$ is connected and hence, $\cg\le \mt {SO}(\oo,B)$.
\end{remark}

Let $\reo$ denote the span of $1$ and $\imo$ be its complement with respect to $B$.
Clearly, for $v\in \oo$, there exists $a\in\reo$ and $b\in\imo$ such that $v=a+b$.
Then, we can define the conjugation map: 
$$\conj v = a-b.$$
Clearly, $\cg$ preserves $1$ and $\imo$.
Hence, conjugation is $\cg$ equivariant.
Using the conjugation, we can express the inner product as $B(u,v)=\re{\conj uv}$.
Also, one can show $\conj{uv}=\conj v\, \conj u$.

Define the cross product by $u\times v = \im{\conj vu}$.
We immediately get
$$u\times v = \conj vu - B(u,v).$$
\begin{proposition}
  The cross product is skew-symmetric.
\end{proposition}
\begin{proof}
  First note that
\begin{eqnarray*}
  u\times u &=& \im{\conj u u} \\
  &=& -\im{\conj{\conj u u}} \\
  &=& -\im{\conj uu} \\
  &=& -u\times u.
\end{eqnarray*}
Thus,
\begin{eqnarray*}
  0&=& (u+v)\times (u+v) \\
  &=& u\times u + u\times v + v\times u + v\times v \\
  &=& u \times v + v \times u
\end{eqnarray*}
which was to be shown.
\end{proof}

Let $\varphi_0(u,v,w) = B(u\times v, w)$. Then, 
\begin{proposition}
$\varphi_0$ is an alternating 3-form on $\imo$.
\end{proposition}
\begin{proof}
Multi-linearity over $\cc$ is a trivial matter to check.
To show that it is an alternating form, first note that $u\times u = 0$ implies $\varphi_0(u,u,v)=0$.

Next, we check $\varphi_0(u,v,u)=0$.
\begin{eqnarray*}
  \varphi_0(u,v,u) &=& B(u\times v, u) \\
  &=& B(\conj vu -B(u,v),u) \\
  \text{(as $u\perp 1$) } &=& B(\conj vu,u) \\
  \text{(by Proposition \ref{prop:ortmult})} &=& B(\conj v,1)Q(u) \\
  \text{(as $v\perp 1$) } &=& 0
 \end{eqnarray*}
 Similarly, (using $\conj v = -v$) we have $\varphi_0(u,v,v)=0$
\end{proof}

Next, we would like to give an alternative description of $\cg$ as the stabilizer of $\varphi_0$ in $\SL(\im\oo)$.
Since $\cg$ acts trivially on $\reo$ and preserves $\imo$, we identify an element of $\cg$ with a linear transformation on $\imo$ and vice versa.
\begin{proposition}
Let $G=\left\{ A\in \SL(\im\oo) | A\pb\varphi_0=\varphi_0 \right\}$.
Then $G=\cg$.
\label{prop:g2stabilizerdefinition}
\end{proposition}
\begin{proof}
First note that for $A\in\cg$,
\begin{eqnarray*}
  A(u\times v) &=& A\im{\conj vu} \\
  &=& \im{A(\conj vu)} \\
  &=& \im{A\conj v Au} \\
  &=& \im{\conj{Av}Au} \\
  &=& (Au)\times (Av).
\end{eqnarray*}
So, 
\begin{eqnarray*}
  \varphi_0(Au,Av,Aw) &=& B(Au\times Av, Aw) \\
  &=& B(A(u\times v),Aw) \\
  &=& B(u\times v, w) \\
  &=& \varphi_0(u,v,w).
\end{eqnarray*}
That is, $\cg\le G$.

For the converse statement, we adapt Bryant's argument in \cite{Bry87}.
First, fix a basis $(e_i)$ for $\imo$ and set $e_0=1$, then it is easy to verify that 
\begin{equation}
(\iota(u)\varphi_0)\wedge(\iota(v)\varphi_0)\wedge\varphi_0 = 6B(u,v)\vol
  \label{eq:metricvol}
\end{equation}
holds for all $u,v\in\imo$ where $\iota(u)\varphi_0$ is the contraction of $\varphi_0$ with $u$, $\vol=e^{12\dots7}=e^1\wedge e^2\wedge\dots\wedge e^7$ and $(e^i)$ is the basis of $\oo^*$ dual to $(e_i)$.
Thus, for $A\in G$, we have
\begin{eqnarray*}
  A\pb\left((\iota(u)\varphi_0)\wedge(\iota(v)\varphi_0)\wedge\varphi_0\right) &=& A\pb\left( 6B(u,v)\vol\right) \\
  \iota(Au)(A\pb\varphi_0)\wedge\iota(Av)(A\pb\varphi_0)\wedge(A\pb\varphi_0)&=& 6B(u,v)A\pb\vol \\
  \iota(Au)\varphi_0\wedge\iota(Av)\varphi_0\wedge\varphi_0&=& 6B(u,v)\vol \\
  6B(Au,Av)\vol &=& 6B(u,v)\vol \\
  B(Au,Av) &=& B(u,v).
\end{eqnarray*}
Thus, $A\in \mt {O}(\oo,B)$.
Since $A$ preserves $B$ and $\varphi_0$, 
\begin{eqnarray*}
  B(A(u\times v), w) &=& B(A(u\times v), AA\inv w) \\
  &=& B(u\times v, A\inv w) \\
  &=& \varphi_0(u,v,A\inv w) \\
  &=& \varphi_0(Au,Av,AA\inv w) \\
  &=& B(Au\times Av, w)
\end{eqnarray*}
for all $w$, that is, the cross product is $G$-equivariant.
For $u,v\in \imo$,
\begin{eqnarray*}
  A(vu) &=& A(\re{vu}+\im{vu}) \\
  &=& A\re{vu}+A\im{vu} \\
  &=& \re{vu}+A(u\times \conj v) \\
  &=& B(\conj v, u)+Au\times A\conj v \\
  &=& B(A\conj v, Au)+\im{\conj{A\conj v} Au} \\
  &=& \re{\conj{A\conj v}Au}+\im{\conj{A\conj v} Au} \\
  &=& \re{A vAu}+\im{Av Au} \\
  &=& AvAu.
\end{eqnarray*}
Now, for $\alpha,\beta\in \reo$,
\begin{eqnarray*}
  A\left( (\alpha+u)(\beta+v) \right) &=& A\left( \alpha\beta+\alpha v+\beta u +uv \right) \\
  &=& A\alpha A\beta + A\alpha Av + A\beta Au + AuAv \\
  &=& A(\alpha+u)A(\beta+v).
\end{eqnarray*}
Hence, $G\le\cg$ and $G=\cg$.
\end{proof}

\subsection{Alternating three-forms in seven-space}\label{sect:altthreeforms}
Let $V$ be a seven-dimensional vector space over $\cc$.
\begin{definition}
An alternating three form $\varphi \in \Lambda^3 V\dual$ is called non-degenerate if for every pair of linearly independent vectors $(u,v)$ there exists $w\in V$ such that
\begin{equation}
  \varphi(u,v,w)\neq 0.
\end{equation}
\end{definition}
\begin{example}
  $\varphi_0$ is a non-degenerate three-form on $\imo$.
\end{example}

If $\varphi$ is a non-degenerate three-form and $u\neq 0$, then $\iota(u)\varphi$ induces a symplectic form on $V/\langle u\rangle$.
Hence, we can choose a symplectic basis on $V/\langle u\rangle$ which we can pull back to $v_i,w_i\in V$ for $i=1,2,3$.
Together these vectors satisfy $\varphi(u,v_i,w_i)=1$ for $i=1,2,3$.
Note that 
\begin{equation}
  \left( \iota(u)\varphi\wedge\iota(u)\varphi\wedge\varphi \right)
  \left( u,v_1,w_1,v_2,w_2,v_3,w_3 \right)
  = \left( (\iota(u)\varphi)^{\wedge3} \right)\left( v_1,w_1,v_2,w_2,v_3,w_3 \right) \neq 0.
  \label{eq:nonVanishingTopForm}
\end{equation}

For the rest of the discussion, we fix an $n$-form $\Omega \in \Lambda^nV\dual$.
Let $x_1,\dots,x_7$ be a basis of $V$ with the dual basis $x^{1},\dots,x^{7}$ satisfying $x^1\wedge\dots\wedge x^7=\Omega$.
Define $b_{ij}$ by
\begin{equation}
  \iota(x_i)\varphi\wedge\iota(x_j)\varphi\wedge\varphi = 6b_{ij}x^{1\dots7}
\end{equation}
where $x^{1\dots7}=x^1\wedge\dots\wedge x^7$.
We think of $(b_{ij})$ as a symmetric matrix.
Using this matrix, we define a symmetric bilinear form by $B(u,v) = u^ib_{ij}v^j$ where $u=u^ix_i$ and $v=v^ix_i$.
\begin{proposition}
  The symmetric bilinear form $B$ is well-defined and non-degenerate.
\end{proposition}
\begin{proof}
First, we consider degeneracy.
Let $u=u^ix_i$ be an eigenvector of $b_{ij}$ with eigenvalue $\lambda$.
Then,
\begin{eqnarray*}
  \iota(u)\varphi\wedge\iota(u)\varphi\wedge\varphi &=& 
  u^iu^j \iota(x_i)\varphi\wedge\iota(x_j)\varphi\wedge\varphi \\
  &=& 6u^iu^jb_{ij}x^{1\dots7} \\
  &=& 6\lambda (u^i)^2 x^{1\dots7}.
\end{eqnarray*}
By (\ref{eq:nonVanishingTopForm}), we know that the left hand side does not vanish.
Hence, $\lambda$ is necessarily non-zero and $(b_{ij})$ is non-degenerate.

Next, we show that $B$ is well-defined.
Let $y_1,\dots,y_7$ be another basis with the dual basis $y^1,\dots,y^7$ such that $y^{1\dots7}=\Omega$.
Define $c_{ij}$ by
\begin{equation}
  \iota(y_i)\varphi\wedge\iota(y_j)\varphi\wedge\varphi = 6c_{ij}y^{1\dots7}.
\end{equation}
Define $L^j_i$ by $y_i=L^{j}_ix_j$.
So, we have
\begin{eqnarray*}
  L_i^kL_j^l\iota(x_k)\varphi\wedge\iota(x_l)\varphi\wedge\varphi &=&  6c_{ij}y^{1\dots7} \\
  = L_i^kL_j^l\left( 6b_{kl} x^{1\dots7}\right) &=& 6c_{ij}y^{1\dots7}.
\end{eqnarray*}
Since $x^{1\dots7}=y^{1\dots7}$, we get
\begin{equation}
  L^k_ib_{kl}L^l_j = c_{ij}.
\end{equation}
Thus, if $u=u^iy_i = u^iL^k_ix_k$ and $v=v^iy_i=v^iL_i^kx_k$,
\begin{eqnarray*}
  C(u,v) &=& u^ic_{ij}v^j \\
  &=& u^iL^k_ib_{kl}L^l_jv^j \\
  &=& B(u,v).
\end{eqnarray*}
\end{proof}

Note that if we scale $\Omega$ by $\lambda$, we scale $B$ by $\lambda\inv$. 
Furthermore, $B$ induces a norm on $\Lambda^nV\dual$.
So by scaling $\Omega$, we may require that the norm of $\Omega$ is 1.
We will implicitly assume this for the rest of the article.
\begin{definition}
  We call the quadruple $(V,\varphi,\Omega,B)$ satisfying (\ref{eq:metricvol}) and $N(\Omega)=1$ (where $N$ is the quadratic form associated to $B$) a $G_2$-(vector) space.
  \label{defn:g2space}
\end{definition}
\begin{remark}
  \label{rmk:g2space}
  One can also define real $G_2$-spaces in a similar manner.
  In fact, over $\rr$, $\varphi$ determines both a metric and a volume form uniquely.
  In that case the metric need not be positive definite.
  The 3-form $\varphi$ is called positive if the metric is positive definite.
\end{remark}

\subsection{The complexification of a \texorpdfstring{$G_2$}{G2}-space}
\label{ssec:complexg2space}
In this section, we exhibit the linear version of some constructions starting with a real 7-dimensional vector space with a positive 3-form $\varphi$.
Although it is possible to do a similar construction with any non-degenerate 3-form, in this section and for the rest of the article, we will focus on positive $\varphi$ (see Remark \ref{rmk:g2space}).
Recall that $\varphi$ determines a (real) $G_2$-space $(V,\varphi,\Omega,g)$. 
Let $V_\cc=V\oplus iV$.
Furthermore, we can extend all of the structures complex linearly and the equation (\ref{eq:metricvol}) continues to hold.
This implies that the complexified three-form is still non-degenerate.
Therefore, we get a (complex) $G_2$-space $(V_\cc,\varphi_\cc,\Omega_\cc,g_\cc)$.

We can also extend $g$ as a hermitian form $h$.
Explicitly, we define
\begin{equation*}
  h(x+iy,z+iw) = g(x,z)+g(y,w) + i\left( g(y,z)-g(x,w) \right).
\end{equation*}
Then, the real part of $h$ is a positive definite metric and the imaginary part is a symplectic form $\omega$ on $V_\cc$.

If $V$ is a half dimensional subspace of $W$ with an almost complex structure $J$ such that $V\oplus JV = W$, we could use $J$ in place of $i$ in the above construction.
This flexibility will be important later on.

\subsection{Compatible structures on a \texorpdfstring{$\cg$}{complex G2}-space}
\label{ssec:compatstr}
K\"ahler geometry is often said to be at the intersection of Riemannian geometry, symplectic geometry and complex geometry because it comes with these three structures that are compatible with each other.
Moreover, any (compatible) two of those structures determines the third one.
At the group level, we can state this as follows
\begin{equation}
  \GL(n,\cc)\cap \mt O(2n)= \mt O(2n)\cap \Sp(2n) = \Sp(2n)\cap \GL(n,\cc) = \U(n),
  \label{eq:kahlerintersection}
\end{equation}
see \cite{MS17}.
Our construction (see subsection \ref{ssec:complexg2space}) of a positive-definite metric $g$, a symplectic form $\omega$ and a (complex) non-degenerate three-form $\varphi_\cc$ from a given (real) non-degenerate three-form $\varphi$ allows us to talk about compatibility between these structures related to $G_2$ geometry.
In this section, we describe this relation for a complex 7-dimensional vector space $(V,J)$.

\begin{definition}
  \label{defn:compatiblelinear}
We say that the triple $(g,\omega,\varphi_\cc)$ is compatible if there is a real 7 dimensional subspace $\Lambda$ of $V$ and a positive $\varphi$ on $\Lambda$ (determining a metric $g'$ on $\Lambda$) such that 
\begin{enumerate}
  \item $V=\Lambda\oplus J\Lambda =: \Lambda_\cc$
  \item $\varphi_\cc$ is the complex linear extension of $\varphi$
  \item $g+i\omega$ is the hermitian extension of $g'$.
\end{enumerate}
In this case, we say they are induced from $(\Lambda,\varphi,J)$.
\end{definition}

Let the stabilizer of $\varphi$ in $\GL(7,\rr)$ be $G_2$ and $\varphi_\cc$ be its complex extension to $\cc^7$.
Then, $A\in G_2$ is of determinant 1, commutes with $i$ and therefore, fixes $\varphi_\cc$.
In other words, $G_2\subset \cg$.
\begin{proposition}
  $$ \cg \cap \U(7) = G_2$$ 
  \label{prop:compactsubgroupofg2c}
\end{proposition}
\begin{proof}
  We have already seen that $ G_2\subset \cg$.
  Since $ G_2\subset \mt{O}(7,\rr)\subset \U(7)$, $ G_2\subset\cg\cap\U(7)$.

  For the converse, first note that $\U(7)\cap \mt O(7,\cc)=\mt O(7,\rr)$ since a matrix whose inverse is both its conjugate transpose and transpose, must be a real matrix.
  Therefore, by Proposition \ref{prop:ort}, $\cg\cap\U(7)\subset \mt O(7,\rr)$.
  So, the intersection consist of real $7\times 7$ matrices preserving $\varphi_\cc$.
  In particular, they preserve $\varphi$ and we get 
$$\cg\cap \mt U(7) = G_2.$$
\end{proof}

Now, using (\ref{eq:kahlerintersection}) and Proposition \ref{prop:compactsubgroupofg2c}, it is easy to see that we have
$$ \cg\cap \mt O(14) = \cg \cap Sp(14) = G_2. $$

We will need the following technical lemma later.
\begin{lemma}
  Given a symplectic form $\omega$ on $\rr^{14}$, a Lagrangian subspace $\Lambda$ and a positive 3-form $\varphi$ on $\Lambda$, let $\mc J(\omega,\Lambda,\varphi)$ be the space of almost complex structures $J$
  such that the triple $(g',\omega',\varphi_\cc)$ induced from $(\Lambda,\varphi,J)$ satisfies
  \begin{enumerate}
    \item $\omega=\omega'$,
    \item $g'|_\Lambda=g$, and
    \item $\varphi_\cc|_\Lambda = \varphi$
  \end{enumerate}
  where $g$ is the metric on $\Lambda$ induced from $\varphi$.
  Then, $\mc J(\omega,\Lambda,\varphi)$ is contractible.
  \label{lem:technical}
\end{lemma}
In fact, we can state this lemma in more general terms.
Then, the proof will follow from Lemma \ref{lem:technicalgeneralcase}.
\begin{lemma}
  Given a symplectic form $\omega$ on $\rr^{14}$, a Lagrangian subspace $\Lambda$ and a metric $g$ on $\Lambda$, let $\mc J(\omega,\Lambda,g)$ be the space of almost complex structures compatible with $\omega$ 
  and $g(x,y)=w(x,Jy)$ for $x,y\in\Lambda$.
  Then, $\mc J(\omega,\Lambda,g)$ is contractible.
  \label{lem:technicalgeneralcase}
\end{lemma}
\begin{remark}
  Lemma \ref{lem:technicalgeneralcase} says that the set of almost complex structures compatible with a given symplectic form and a fixed metric on some Lagrangian subspace is contractible.
\end{remark}
\begin{proof}
  First, we choose an orthonormal basis $\left\{ e_i \right\}$ for $\Lambda$ and extend it to $\omega$-standard basis $\left\{ e_i,f_i \right\}$.
  So, $\omega_0 = \sum_{i=1}^n e^i\wedge f^i$.
  We think of $J\in \mc J(\omega,\Lambda,g)$ as an $2n\times 2n$ matrix with respect to this basis.
  Note that $J\in \mc J(\omega,\Lambda,g)$ if and only if
  \begin{enumerate}
    \item $J^2=-\id_{2n}$,
    \item $J^\T J_{2n}J=J_{2n}$ where $J_{2n}=\begin{pmatrix} 0 & -\id_n \\ \id_n & 0 \end{pmatrix}$
    \item $-J_{2n}J=\begin{pmatrix} \id_n & B \\ B^\T & C \end{pmatrix}$ is symmetric positive definite.
  \end{enumerate}

  Let $P=-J_{2n}J$.
  Note that
  \begin{align*}
    P^\T J_{2n}P &= -J^\T J_{2n}J_{2n}J_{2n}J \\
    &= J^\T J_{2n} J \\
    &=J_{2n}.
  \end{align*}
  This implies $C= \id + BB^\T$.
  Define the path $P_t =\begin{pmatrix} \id_n & tB \\ tB^\T & I+t^2BB^\T \end{pmatrix}$.
  Clearly, $P_t^\T = P_t$.
  Next, we check if $P_t$ is a symplectic matrix.
  \begin{align*}
    & 
    \begin{pmatrix}
      \id_n & tB^\T \\
      tB & \id_n + t^2 BB^\T
    \end{pmatrix}
    \begin{pmatrix}
      0 & -\id_n \\
      \id_n & 0
    \end{pmatrix}
    \begin{pmatrix}
      \id_n & tB \\
      tB^\T & \id_n + t^2 BB^\T
    \end{pmatrix} \\
    &=  
    \begin{pmatrix}
      \id_n & tB^\T \\
      tB & \id_n + t^2 BB^\T
    \end{pmatrix}
    \begin{pmatrix}
      -tB^\T & -\id_n-t^2BB^\T \\
      \id_n & tB
    \end{pmatrix}
    \\
    &= 
    \begin{pmatrix}
      0 & -\id_n \\
      \id_n & 0
    \end{pmatrix}
  \end{align*}
    Therefore, $P_t$ is invertible for all $t$.
    Since it is always symmetric and at $t=0$ (or $t=1$) it is positive definite, $P_t$ is positive definite for all $t$.
    Hence, $J_{2n}P_t$ is a path in $\mc J(\omega,\Lambda, g)$ from $J_{2n}$ to $J$.
    Clearly, the path depends continuously on $J$.
\end{proof}
\begin{proof}[Proof of Lemma \ref{lem:technical}]
  The first two properties imply that $\mc J(\omega,\Lambda,\varphi)=\mc J(\omega,\Lambda,g)$ where $g$ is the metric induced from $\varphi$ on $\Lambda$.
  Thus, Lemma \ref{lem:technicalgeneralcase} shows that it is contractible.
  The third property is trivially satisfied by definition of complex linear extension.
\end{proof}

\section{Grassmannians}
In this section, we consider various Grassmannians related to our discussions.

\subsection{Associative Grassmannian}
In this section we focus on a discussion of the associative Grassmannian over $\cc$.
The reader can consult to \cite{AC15} for a more comprehensive description of this variety or to \cite{HL82,AK16,AS08} for more details on the associative Grassmannian over $\rr$ (the Cayley version is discussed in \cite{Yil17}).

Using octonionic multiplication one can define {\it an associator bracket} as follows
\begin{equation}
  [u,v,w]=\frac{1}{2}\left( u(vw)-(uv)w \right).
\end{equation}
Clearly, this bracket measures whether given three octonions satisfy associativity or not.
A three-dimensional subspace of $\imo$ on which the associator bracket vanishes is called {\it associative}.
The space of all associative planes is called the associative Grassmannian and it is denoted by Gr($\varphi$).

\begin{remark}
  Our definition of associatives differs from that of \cite{AC15}.
  In fact, we will often require $B$ to be non-degenerate on an associative plane $L$.
  This is the convention of \cite{AC15} and only with this convention, it is possible to find a $B$-orthonormal basis of $L$ on which $\varphi$ evaluates to $\pm1$.
  We explicitly state so whenever we require this condition.
\end{remark}

We denote by Gr$^\cc(k,n)$ the complex Grassmannian of $k$-planes in $n$-dimensional space and by Gr$^\rr(k,n)$ the real Grassmannian of $k$-planes in $n$-dimensional space.
Clearly, after choosing identifications $\oo\cong\cc^8\cong\rr^{16}$, we have Gr$(\varphi)\subset$ Gr$^\cc(3,7)\subset$ Gr$^\rr(6,14)$.

It turns out that the associator bracket is the imaginary part of a triple cross product defined as follows
\begin{equation}
  u\times v\times w =  (u\conj v)w-(w\conj v)u
\end{equation}
for all $u,v,w\in\oo$.
More precisely, $\im{u\times v\times w} = [u,v,w]$ for $u,v,w\in\imo$.

\begin{proposition}
  For $u,v,w\in\imo$,
  \begin{equation*}
    [u,v,w] = u\times(v\times w) + B(u,v)w-B(u,w)v.
  \end{equation*}
  \label{prop:bracketalt}
\end{proposition}
\vspace{-0.2in}
\begin{proof}
  Since octonions are alternative (i.e. any subalgebra generated by two elements is associative), we immediately see that the associator bracket is alternating.
  We denote the right hand side by $R(u,v,w)$.
  Clearly, $R(u,v,v)=0$.

  Next, we show $R(u,u,v)=0$.
  \begin{eqnarray*}
    u\times(u\times v) &=& \im{ \conj{\im {\conj vu} }u} \\
    &=& \im{\im{vu}u} \\
    &=& \im{(vu-\re{vu})u} \\
    &=& (vu-\re{vu})u - \re{(vu-\re{vu})u} \\
    &=& vu^2-\re{vu}u - \re{vu^2} \\
    &=& \im{vu^2} - \re{vu}u \\
    &=& -\im{vB(u,u)} - B(\conj v,u)u \qquad \mt{since $u^2=-B(u,u)$} \\
    &=& -B(u,u)v + B(u,v)u.
  \end{eqnarray*}

Since both sides of the equality are alternating, it is enough to check the equality on orthonormal triples.
Note that both sides of the equation are $\cg$-equivariant, that is $[gu,gv,gw]=g[u,v,w]$ and $R(gu,gv,gw)=gR(u,v,w)$ for $g\in\cg$.
Furthermore, $\cg$ acts transitively on $\cg$-triples, so we actually only need to consider two types of basis vectors of the form $(i,j,x)$ for $x = k,l$.
This can now be easily verified by using the definitions and octonionic multiplication table.

\end{proof}

Recall that $L\in$ Gr$^\ff(k,n)$ has a neighborhood which can be identified with \linebreak $\Hom{\ff}{L}{\ff^n/L}$ which also gives us the following identification 
\begin{equation*}
T_{L}\mt{Gr}^\ff(k,n) \cong \Hom{\ff}{L}{\ff^n/L} \cong L\dual\otimes \ff^n/L
\end{equation*}
for $\ff=\rr, \mt{ or }\cc$.
Let $B$ be a non-degenerate symmetric bilinear form on $\ff^n$ and $L$ be a subspace of $\ff^n$.
With a little abuse of terminology, we say that $L$ is non-degenerate if $B\big|_{L\times L}$ is non-degenerate.
If $L$ is non-degenerate, then we may identify $\ff^n/L$ with $L^\perp=\left\{ v\in\ff^n\; | \; B(u,v)=0 \;\mt{ for all } u\in L \right\}$.
Thus, $T_L\mt{Gr}^\ff(k,n)\cong \Hom{\ff}{L}{L^\perp}$.

Let $\ee$ denote the tautological vector bundle over Gr$^\ff(k,n)$, i.e. the fiber $\ee_L$ over $L\in\mt{Gr}^\ff(k,n)$ is $L$ itself.
Let $\vv=\ee^\perp$.
Over non-degenerate $L$, $\ff^n=L\oplus L^\perp$.
Therefore, over the open dense subset $N\subset\mt{Gr}^\ff(k,n)$ of non-degenerate $k$-planes we have the following isomorphism
\begin{equation*}
T\mt{Gr}^\ff(k,n)\big|_N \cong \left( \ee\dual\otimes \vv \right)\big|_N.
\end{equation*}

Next, we would like to prove the analogue of Lemma 5 in \cite{AS08}.
\begin{lemma}
Let $L\in\mt{Gr}(\varphi)$ be non-degenerate and $L=\langle e_1,e_2,e_3 = e_1\times e_2\rangle$ be an orthonormal basis for $L$. Then
\begin{equation*}
T_L\mt{Gr}(\varphi)=\left\{ \sum_{j=1}^3 e^{j}\otimes v_j \in \ee\dual \otimes \vv \;\;\Big|\; \sum e_j\times v_j =0 \right\}.
\end{equation*}
\label{lem:tangentofassociate}
\end{lemma}
\begin{proof}
The proof is virtually the same as in \cite{AS08} except for the fact that we are now in the holomorphic category.
So, the paths we consider will be holomorphic maps from the unit disk.
Nevertheless, we reproduce the proof here for the sake of completeness.

We first identify $\cc^7/L$ with $L^\perp$ using orthogonal projection.
Let $\gamma$ be a path in $\Hom{\cc}{L}{L^\perp} \; \cap $ Gr$(\varphi)\subset\mt{Gr}^\cc(3,7)$ with $\gamma(0)=L$.
Set $e_i(t)=e_i+\gamma(t)e_i$ for $i=1,2,3$.
Since $\gamma(t)\in \mt{Gr}(\varphi)$, we have $\left[ e_1(t),e_2(t),e_3(t) \right]=0$ for all $t$.
Taking derivative of both sides and evaluating at $t=0$, we get
\begin{align*}
0=& \left[ \dot e_1(0),e_2(0),e_3(0) \right] + \left[ e_1(0),\dot e_2(0),e_3(0) \right] + \left[ e_1(0),e_2(0),\dot e_3(0) \right] \\
=& \left[ \dot e_1(0),e_2,e_3 \right] + \left[ e_1,\dot e_2(0),e_3 \right] + \left[ e_1,e_2,\dot e_3(0) \right].
\end{align*}
Clearly, $\dot e_i(0)\in L^\perp$.
Therefore, by Proposition \ref{prop:bracketalt}, $\left[ \dot e_i(0),e_j,e_k \right]=\dot e_i(0)\times(e_j\times e_k)$. 
Further, $\dot e_i(0)\times(e_j\times e_k) = \dot e_i(0)\times e_i$ for a cyclic permutation $(i,j,k)$ of $(1,2,3)$. 
Thus, we have
\begin{equation*}
\sum_{j=1}^3 e_j\times \dot e_j(0).
\end{equation*}
\end{proof}

\subsection{Isotropic Grassmannian}
In this section, first we recall some basic definitions and facts of symplectic topology (see \cite{MS17}).
Then, we define and investigate the Grassmannian of isotropic planes.

Let $\omega=\sum_i e^i\wedge f^i$ be the standard symplectic form on $\rr^{2n}=\rr^{n}\oplus i\rr^{n}$
where $\left\{ e_i \right\}$ is the standard basis of $\rr^n$, $f_i=\sqrt{-1}e_i$ and $\left\{ e^i,f^i \right\}$ is the dual basis of $\left\{ e_i,f_i \right\}$.
For a subspace $L$, we define its symplectic complement $L^\omega$ to be 
\begin{equation*}
  L^\omega =\left\{ v\in \rr^n\oplus i\rr^n \; | \; w(u,v)=0 \text{ for all } u\in L \right\}.
\end{equation*}

\begin{definition}
  A $k$-plane $L$ is called
  \begin{itemize}
    \item isotropic if $L \subset L^\omega$,
    \item coisotropic if $L^\omega \subset L$, 
    \item symplectic if $L\cap L^\omega = \left\{ 0 \right\}$,
    \item Lagrangian if $L=L^\omega$.
  \end{itemize}
  \label{defn:sympspaces}
\end{definition}

Next, we give a description of the set $I_k$ of isotropic $k$-planes for $k\le n$.
Since the symplectic complement of an isotropic plane is coisotropic and that of a coisotropic plane is isotropic, the set of coisotropic $(n-k)$-planes will be isomorphic to $I_k$.
Clearly, the set of all Lagrangian subspaces is $I_n$.

Given an isotropic $k$-plane $L$, we choose an orthonormal basis $\left\{ z_j \right\}=\left\{ x_j+iy_j \right\}$ ($1\le j \le k$) of $L$.
We express these vectors in the standard basis $\left\{ e_i,f_i \right\}$ of $\rr^n\oplus i\rr^n$ i.e.  $x_i = x_i^je_j$ and $y_i = y_i^jf_j$.
Then, we form the following $2n\times k$ matrix
\begin{equation*}
  \begin{pmatrix}
    X\\Y
  \end{pmatrix}
  =
  \begin{pmatrix}
    x_i^j\\y_i^j
  \end{pmatrix}.
\end{equation*}
The fact that $\left\{ z_j \right\}$ is an orthonormal basis is equivalent to 
\begin{equation}
  \label{eq:unitframeeq1}
  X^\T X+Y^\T Y = I_k.
\end{equation} 

Note that in the above basis, the matrix representation of $\omega$ is
\begin{equation*}
  \begin{pmatrix}
    0 & I_n \\
    -I_n & 0
  \end{pmatrix}.
\end{equation*}
Thus, the fact that $L$ is isotropic is equivalent to
\begin{equation}
  \label{eq:unitframeeq2}
  X^\T Y = Y^\T X.
\end{equation}
Combining these two equations, we see that $\left\{ z_j \right\}$ is a unitary basis (with respect to the standard hermitian structure on $\rr^n\oplus i\rr^n=\cc^n$) of $L$.
So, the projection map from unitary matrices to the span of their first $k$-columns is a surjective map onto $I_k$.
This also shows that $U(n)$ acts transitively on all isotropic $k$-planes.
We set the ``standard'' isotropic plane $L_0$ to be the span of $\left\{ e_1,\dots,e_k \right\}$.
Then, it is easy to see that its stabilizer consists of matrices of the form 
\begin{equation*}
  \begin{pmatrix}
    A & 0 \\
    0 & B
  \end{pmatrix}
  \in U(n)
\end{equation*}
for $A\in O(k)$ and $B\in U(n-k)$.
Hence,
\begin{equation}
  I_k = U(n)\big/ (O(k)\times U(n-k)).
  \label{eq:isotropicgrassmannian}
\end{equation}
In particular, the set of Lagrangian subspaces is isomorphic to
\begin{equation*}
  I_n = U(n)/O(n).
\end{equation*}

Next, we prove the analogue of Lemma \ref{lem:tangentofassociate} for $I_k\subset Gr^\rr(k,2n)$.
Before we state the lemma, recall that any orthonormal basis of an isotropic subspace $L$ can be extended to a orthonormal $\omega$-standard basis for $\rr^{2n}$, see \cite{MS17}.
Set $\ee$ to be the tautological bundle over $Gr^\rr(k,2n)$ and $\vv=\ee^\perp$.
\begin{lemma}
  Let $L\in I_k$ and $L=\langle e_1,\dots,e_k\rangle$ be an orthonormal basis for $L$. 
  Extend this basis to an orthonormal $\omega$-standard basis $\left\{ e_i,f_i \right\}$.
  Then
\begin{equation*}
  T_L\mt I_k=\left\{ \sum_{j=1}^k e^{j}\otimes v_j \in \ee\dual \otimes \vv \;\;\Big|\; \omega(e_i,v_j)=\omega(e_j,v_i) \right\}.
\end{equation*}
\label{lem:tangentofisotropic}
\end{lemma}
\begin{proof}
  Let $\phi_t:L\to L^\perp$ be a path of isotropic planes, i.e. $\phi(0)=0$ and
  $\omega|_{L_t}\equiv 0$ for all $t$ where $L_t=\langle e_i(t)\rangle_{i=1}^k$ and $e_i(t) = e_i + \phi_t(e_i)$ for $1\le i \le k$.
  Set $\dot e_i =\frac{d}{dt}|_{t=0}e_i(t)$.
  Since $\omega(e_i(t),e_j(t))=0$ for all $t$ after taking derivatives and plugging $t=0$, we get
  \begin{align*}
    \omega(\dot e_i,e_j) + \omega(e_i,\dot e_j) =  0.
  \end{align*}
\end{proof}

\subsection{Isotropic associative Grassmannian}
\label{ssec:iag}
In this section, we define a new type of Grassmannian called the isotropic associative Grassmannian.
The planes in this Grassmannian sit in a specific intersection of two different geometries, namely $G_2$-geometry and symplectic geometry.
Arbitrary choices of a symplectic form and (complex) non-degenerate three-form (on $\cc^7$) may result in different intersections of (real) associative and isotropic planes.
However, the construction of subsection \ref{ssec:complexg2space} determines a particular way to intersect the two geometries.

First, we fix a particular octonion algebra $\oo$ as in \cite{Yil17}.
Then, $(\im\oo,\varphi_\cc,\Omega_\cc,B)$ is actually the complexification of a real $G_2$ space ($\rr^7,\varphi,\Omega,g$) with a positive definite inner product and the standard symplectic structure $\omega$ is compatible with $\varphi_\cc$.

\begin{definition}
\label{defn:isotropicassocplane}
Let $L$ be a (real) 3-dimensional subspace of $\im \oo = \cc^7$.
We call $L$ isotropic associative if 
\begin{enumerate}
\item $\omega\big|_L\equiv 0$, and
\item $[u,v,w]_\cc=0$ for $u,v,w\in L$.
\end{enumerate}
We denote the space of all isotropic associative planes by $I^\varphi_3\subset Gr^\rr(3,14)$.
\end{definition}

The following lemma describing the tangent space of $I^\varphi_3$ in $Gr^\rr(3,14)$ at a (real) associative plane can be proved by a combination of Lemma \ref{lem:tangentofassociate} and Lemma \ref{lem:tangentofisotropic}.
\begin{lemma}
  Let $L=\langle e_1,e_2,e_3\rangle$ be an associative plane in $\rr^7$.
  Then its natural embedding in $\cc^7=\im{\oo}$ is an isotropic associative.
Denote the (real) tautological bundle over $Gr^\rr(3,14)$ by $\ee$.
Also, set $\vv=\ee^{\perp B}$.
Then $\ee^{\perp g}=J\ee\oplus \vv$ and
\begin{eqnarray}
  T_{L}I^\varphi_3&=&\left\{ \sum_{i=1}^3 e^i\otimes (f_i+v_i) \in \ee\dual\otimes_\rr (J\ee\oplus\vv) \; \right.
  \\ & &  \quad\; | \; \left.\sum e_i\times v_i =0 \text{ and } \omega(e_i,f_j)=\omega(e_j,f_i) \right\}
  \nonumber
\end{eqnarray}
  \label{lem:tangentofia}
\end{lemma}


\subsection{\texorpdfstring{$B$}{B}-Real associative Grassmannian}
It is possible to define a different notion of a ``real'' part of complex associative planes.
We call them $B$-real associative planes.

\begin{definition}
\label{defn:realassocplane}
Let $L$ be a (real) 3-dimensional subspace of $\im \oo$.
Set $g=\re{B}$ and $w=-\im{B}$.
We call $L$ $B$-real associative if 
\begin{enumerate}
\item $g\big|_L$ is positive definite (in particular non-degenerate),
\label{racond1}
\item $w\big|_L\equiv 0$, and
\label{racond2}
\item $[u,v,w]_\cc=0$ for $u,v,w\in L$.
\label{racond3}
\end{enumerate}
We denote the space of all $B$-real associative planes by $RA\subset Gr^\rr(3,14)$.
\end{definition}

\begin{remark}
  In Definition \ref{defn:realassocplane}, (\ref{racond1}) is a technical condition which simplifies the statements and discussions later on.
\end{remark}

\begin{remark}
\label{rmk:gorthborth}
If $L$ is a real associative plane, $g(u,v)=B(u,v)$ for $u,v\in L$.
Hence, a $g$ orthonormal frame is also a $B$ orthonormal frame.
Therefore, the natural complexification $L_\cc = L\oplus JL$ of $L$ is in $Gr(\varphi_\cc)$.
Under this map, we obtain the following disk bundle
\begin{equation*}
D^3\to RA \to Gr(\varphi_\cc).
\end{equation*}
\end{remark}

Next, we determine the tangent space of $RA$.
\begin{lemma}
Let $L=\langle e_1,e_2,e_3\rangle$ be a $B$-real associative plane in $\im\oo$.
By $\ee$, we denote the (real) tautological bundle over $Gr^\rr(3,14)$.
Set $\vv=\ee^{\perp B}$.
Then $\ee\oplus\ee^{\perp g}$ is a splitting of the trivial $\rr^{14}$ bundle over $RA$, $\ee^{\perp g}$ decomposes as $J\ee\oplus \vv$ and
\begin{equation}
T_{L}RA(\im\oo) = \left\{ \sum_{i=1}^3 e^i\otimes (f_i+v_i) \in \ee\dual\otimes_\rr (J\ee\oplus\vv) \; | \; \sum e_i\times v_i =0 \text{ and } (f_{ij})\in\mathfrak{so}(3) \right\}
\end{equation}
where by $f_{ij}$ we mean the $j^{th}$ component of $f_i$ with respect to the basis $\left\{ Je_1,Je_2, Je_3\right\}$.
\label{lem:tangentofrealassocplanes}
\end{lemma}
\begin{proof}
Since $g$ is non-degenerate on $L$, it is clear that $\ee_L\oplus\ee_L^{\perp g}$ spans $\rr^{14}$ for $L\in RA$.
This gives a global splitting $\ee\oplus\ee^{\perp g}$ of the trivial $\rr^{14}$ bundle over $RA$.

Furthermore, it is clear that $\vv=\ee^{\perp B}$ is a subbundle of $\ee^{\perp g}$.
Since $B(u,v)=g(u,v)-ig(Ju,v)$ and $\omega(u,v)=g(Ju,v)$ restricted to $L$ vanishes, $J\ee$ is a subbundle of $\ee^{\perp g}$.
Also, $J\ee\cap\vv = 0$.
So, for dimension reasons, $J\ee\oplus\vv=\ee^{\perp g}$.

Recall that the tangent space $T_LGr^\rr(3,14)=L\dual\otimes L^{\perp g}$.
The above discussion gives us the refinement $L^{\perp g}=JL\oplus L^{\perp B}$.
So, after choosing an orthonormal frame $\left\{ e_1,e_2,e_3 \right\}$ for $L$, we have $X=\left( \sum_{i=1}^3 e^i\otimes(f_i+v_i) \right)\in T_LRA(\im\oo)\subset T_LGr^\rr(3,14)$ where $\left\{ e^1,e^2,e^3 \right\}$ is the basis dual to $\left\{ e_1,e_2,e_3 \right\}$, $f_i\in JL$ and $v_i\in L^{\perp B}$.

A path $\gamma$ in $Gr^\rr(3,14)$ with $\gamma(0)=L$ is (locally) given by paths $f_i$ in $JL$ and $v_i$ in $L^{\perp B}$ with $f_i(0)=0$ and $v_i(0)=0$ for $i=1,2,3$.
Set $\gamma_i(t) = e_i+f_i(t)+v_i(t)$.
So, we have
\begin{equation*}
\gamma(t)=\langle \gamma_i(t) \rangle_{i=1}^3.
\end{equation*}

Assuming $\gamma(t)\in RA$ for all (small) $t$, we check the conditions imposed on $\gamma'(0)$.
Positive-definiteness is an open condition and hence, it does not introduce any condition on $\gamma'(0)$.

Next, by condition \ref{racond2}, $\gamma$ satisfies
\begin{equation*}
\omega(\gamma_i(t),\gamma_j(t))=g(J\gamma_i(t),\gamma_j(t))=0
\end{equation*}
for all $1\le i,j \le 3$.
Applying $\frac{d}{dt}\big|_{t=0}$ to both sides, we get
\begin{align*}
0 &= g(J\dot f_i(0)+J\dot v_i(0),e_j) + g(Je_i,\dot f_j(0)+\dot v_j(0)) \\
&= g(J\dot f_i(0),e_j) + g(Je_i,\dot f_j(0)) \\
&= g(\dot f_i(0),Je_j) + g(Je_i,\dot f_j(0)) \\
&= g(Je_j,\dot f_i(0)) + g(Je_i,\dot f_j(0)).
\end{align*}
Since $\left\{ Je_1,Je_2,Je_3 \right\}$ is an orthonormal basis for $L$, the last equality gives us
\begin{equation*}
\dot f_i^{j}(0)=-\dot f_j^i(0)
\end{equation*}
where $\dot f_i(0)=\sum_{j=1}^3{\dot f}_i^j(0)Je_j$.

Finally, by condition \ref{racond3}, $\gamma$ satisfies
\begin{equation*}
0 = \left[ \gamma_1(t),\gamma_2(t),\gamma_3(t) \right].
\end{equation*}
Applying $\frac{d}{dt}\big|_{t=0}$ to both sides, we get
\begin{equation}
0 = [\dot f_1(0)+\dot v_1(0), e_2, e_3] + [e_1, \dot f_2(0)+\dot v_2(0), e_3] + [e_1, e_2, \dot f_3(0)+\dot v_3(0)].
\label{eq:tobecontinued}
\end{equation}
Note that since $L$ is real associative and $\dot f_i(0)\in L_\cc$, we have
\begin{equation*}
[\dot f_i(0),e_j,e_k]=0
\end{equation*}
for all $i,j,k$.
So, the equation (\ref{eq:tobecontinued}) simplifies to
\begin{align*}
0 &= [\dot v_1(0), e_2, e_3] + [e_1, \dot v_2(0), e_3] + [e_1, e_2, \dot v_3(0)].
\end{align*}
The rest of the proof is as in the proof of Lemma \ref{lem:tangentofassociate} and we get
\begin{equation*}
0=\sum_{i=1}^3 e_i\times \dot v_i(0).
\end{equation*}
\end{proof}

\subsection{Diagram of all Grassmannians}
In this section, we give a diagram of all the relevant Grassmannians showing various maps between them. Whenever we have a map from a real $k$-plane Grassmannian to a complex $k$-plane Grassmannian, the complexification is well defined and we implicitly complexify. Let $F$ be $\rr$ or $\cc$. Let us recall $G_{k}(F^{n})$ is the Grassmannian of $k$ planes in $F^{n}$, and  $G^{\varphi}_{3}(F^{7})$ is the associative grassmannian $3$-planes in $F^{7}$, and $I_{3}(\cc^7)$  is the isotropic Grassmannian of $3$-planes in $\cc^{7}$. Then we have the following inclusions and fibrations (the vertical maps):

\begin{center}
  \begin{tikzpicture}
    \matrix (m) [matrix of math nodes, minimum width=2em, row sep=1.5em, column sep=2em]
    { 
    &                    &                    & G_3(\rr^{14})        & \frac{SO(14)}{SO(3)\times SO(11)} & \\
  &                    & I_3^\varphi(\cc^7) & I_3(\cc^7) & \frac{U(7)}{O(3)\times U(4)}      & \frac{U(3)}{O(3)}\\
      \frac{G_2}{SO(4)}  & G_3^\varphi(\rr^7) & G_3^\varphi(\cc^7) & G_3(\cc^7 )        & \frac{U(7)}{U(3)\times U(4)} & \\
  & G_3(\rr^7)              & \frac{SO(7)}{ SO(3)\times SO(4)}               & G_3^\cc(\rr^7)                   & \frac{SO(7,\cc)}{SO(3,\cc)\times SO(4,\cc)}                                         & \\ };
    \path[-stealth]
    (m-3-2) edge [right hook->] (m-2-3)
    edge [right hook->] (m-3-3) 
    edge [double,-] (m-3-1) 
    edge [right hook->] (m-4-2) 
    (m-4-2) edge [double,-] (m-4-3)
    (m-4-3) edge [right hook->] (m-4-4)
    (m-4-4) edge [double,-] (m-4-5)
    edge [right hook->] (m-3-4)
    (m-3-3) edge [right hook->] (m-3-4)
    (m-2-3) edge [right hook->] (m-2-4)
    edge [right hook->] (m-1-4)
    edge (m-3-3)
    (m-2-4) edge [right hook->] (m-1-4)
    edge [double,-] (m-2-5)
    edge (m-3-4)
    (m-3-4) edge [double,-] (m-3-5)
    (m-1-4) edge [double,-] (m-1-5)
    (m-2-6) edge (m-2-5)
    ;
  \end{tikzpicture}
\end{center}

\section{\texorpdfstring{$\cg$}{Complex G2}-manifolds}
In this section, we fix a particular octonion algebra as in subsection \ref{ssec:iag}.
Let $\oo$ be a copy of $\cc^8$ generated by $e_0=1,e_1,\dots,e_7$ forming an orthonormal basis (with respect to the standard $B$) and $e^0,\dots,e^7$ be its dual basis.
The non-degenerate three-form $\varphi_0$ is given by
\begin{equation}
  \varphi_0 = e^{123}-e^{145}-e^{167}-e^{246}+e^{257}-e^{347}-e^{356}
\end{equation}
in this basis where $e^{ijk}=e^i\wedge e^j\wedge e^k$.
Let $u\in\imo$ be of norm 1.
We define $\Omega_u$ by
\begin{equation}
  6\Omega_u = \iota(u)\varphi_0\wedge\iota(u)\varphi_0\wedge\varphi_0.
\end{equation}
Surprisingly, this definition is independent of $u$.
Indeed, for any $u,v\in\imo$ of norm 1, there is a linear transformation $A\in\cg$ such that $Au=v$.
Thus,
\begin{eqnarray*}
  \Omega_u &=& A\pb (\Omega_u) \\
  &=& A\pb(\iota(u)\varphi_0\wedge\iota(u)\varphi_0\wedge\varphi_0) \\
  &=& \iota(Au)A\pb\varphi_0\wedge\iota(Au)A\pb\varphi_0\wedge A\pb\varphi_0 \\
  &=& \iota(Au)\varphi_0\wedge\iota(Au)\varphi_0\wedge\varphi_0 \\
  &=& \iota(v)\varphi_0\wedge\iota(v)\varphi_0\wedge\varphi_0 \\
  &=& \Omega_v.
\end{eqnarray*}
Let $\Omega_0=\Omega_u$ for any $N(u)=1$.
Evaluating $\Omega_u$ for $u=e_1$, we see that $\Omega_0=e^{1\dots7}$.

For the rest of this section, we identify $(x_1,\dots,x_7,y_1,\dots,y_7)\in\rr^{14}$ with $\Sigma x_je_j+y_jie_j \in\rr\langle e_1,\dots,e_7\rangle\oplus i\rr\langle e_1,\dots,e_7\rangle=\imo$.
This allows us to identify $\cg$ as a subgroup of $\GL(14,\rr)$.
Let $M$ be a 14-manifold and $m\in M$.
An $\rr$-isomorphism $L:\rr^{14}\to T_mM$ is called a frame over $m$ and the frame bundle of $M$ is the collection of all frames as $m$ varies over $M$.
\begin{definition}
  A (real) 14-dimensional manifold $M$ is called an (almost) $\cg$-manifold if its frame bundle admits a reduction to a principal $\cg$-bundle.
\end{definition}
\begin{proposition}
  A $\cg$-manifold $M$ naturally has the following structures
  \begin{itemize}
    \item an almost complex structure $J\in\Gamma(M;\mt{End}(TM))$
    \item a $\cc$-linear three-form $\varphi\in\Omega^3(M;\cc)$
    \item a $\cc$-linear seven-form $\Omega\in\Omega^7(M;\cc)$
    \item a symmetric bilinear form $B\in\Gamma(M;\mt{S}^2(TM)\otimes\cc)$
    \item two signature $(n,n)$ pseudo-Riemannian metrics $g_1=\re B$ and $g_2=\im B$.
  \end{itemize}
\end{proposition}
\begin{proof}
  Since $\cg$ preserves each one of these structures, one may pull them back onto $M$ by using a $\cg$-frame.
\end{proof}
Next, we reformulate the above definition.
Since $\cg$ is the stabilizer of $\varphi_\cc$ in $\mt{SL}(\im\oo)$ (by Proposition \ref{prop:g2stabilizerdefinition}), one may also use the following definition of (almost) $\cg$-manifolds.
\begin{definition}
  A (real) 14 dimensional manifold $(M,J,\varphi,\Omega)$ with an almost complex structure $J$, a $\cc$-multilinear three form $\varphi$ and a $\cc$-multilinear seven-form $\Omega$ is called an (almost) $\cg$-manifold if for every $m\in M$, there is an $\rr$-linear isomorphism $(T_mM,J,\varphi,\Omega)\cong(\imo,i,\varphi_0,\Omega_0)$.
\end{definition}

Let $Y$ be a (real) 3-dimensional submanifold of a $\cg$-manifold $M$.
If $T_pY$ is a $B$-real associative plane in $T_pM$ for every $p$, then $Y$ is called a $B$-real associative submanifold.
In order to define isotropic associative submanifolds, one needs an (almost) symplectic structure compatible with $\cg$ structure.
We say that a symplectic structure is compatible with $\cg$ structure if they are point-wise compatible in the sense of subsection \ref{ssec:compatstr}.
We call a (real) 3-submanifold $Y$ isotropic associative submanifold if $T_pY$ is an isotropic associative plane in $T_pM$ for every $p$.

\section{Complexification of a \texorpdfstring{$G_2$}{G2} manifold}
In this section, we give two examples of $\cg$ manifolds with a compatible (almost) symplectic structure.
We start with a usual $G_2$ manifold and construct two different $\cg$ manifold structures on its cotangent bundle.

Our first construction is as follows.
Let $\left(M,\varphi\right)$ be a (real) 7-dimensional $G_2$ manifold.
Recall that $M$ is naturally equipped with a Riemannian metric $g$ and a volume form $\Omega$ satisfying 
\begin{equation}
\iota(u)\varphi\wedge\iota(v)\varphi\wedge\varphi = 6g(u,v)\Omega.
\label{eq:phimetvol}
\end{equation}
We can think of the Levi Civita connection on the cotangent bundle as a horizontal distribution and hence, it induces the isomorphism
\begin{equation}
T_{\alpha}T\dual M \cong T_pM\oplus T\dual_pM
\label{eq:tancotcanonicsplit}
\end{equation}
where $\alpha\in T\dual_pM$ and $p\in M$.
To define an almost complex structure on $TT\dual M$, we view the metric as a vector bundle isomorphism $g:TM\to T\dual M$ and we set
\begin{equation}
J(X+\beta) = -g\inv(\beta)+g(X)
\label{eq:cotacs}
\end{equation}
for $(X,\beta)\in T_pM\oplus T\dual_pM = T_\alpha T\dual M$.
Clearly, $J^2=-\id_{TT\dual M}$.

Next, we ``extend $\varphi$ complex linearly'' to $TT\dual M$, i.e.
we define $\varphi_\cc$ to be the unique $\cc$-valued $3$-form satisfying
\begin{enumerate}
\item $\varphi_\cc(X,Y,Z)= \varphi(X,Y,Z)$ and
\item $\varphi_\cc( J(X), Y, Z) = i \varphi(X,Y,Z)$
\label{eq:cotphi}
\end{enumerate}
for horizontal vectors $X,Y,Z$; where we identify $T_pM$ with horizontal part of $T_\alpha T\dual M$ using (\ref{eq:tancotcanonicsplit}).
Similarly, we extend $g$ and $\Omega$ complex linearly and we denote the complexifications by $B$ and $\Omega_\cc$, respectively.
Then, from (\ref{eq:phimetvol}), we immediately get
\begin{equation}
\iota(\xi)\varphi_\cc\wedge\iota(\varepsilon)\varphi_\cc\wedge\varphi_\cc = 6B(\xi,\varepsilon)\Omega_\cc
\label{eq:phimetvolc}
\end{equation}
for $\xi,\varepsilon\in TT\dual M$.
Note that $B$ is non-degenerate and $\Omega_\cc$ is a non-vanishing complex volume form.
Therefore, by (\ref{eq:phimetvolc}), $\varphi_\cc$ is non-degenerate.
We extend $g$ as a hermitian form $h$ as well.
So, $\re h$ is a positive definite metric and $\omega=\im h$ is an almost symplectic form on $T\dual M$.
More explicitly,
\begin{equation*}
  \omega(X+\alpha, Y+\beta) = \alpha(Y)-\beta(X).
\end{equation*}
From the construction it is clear that $\varphi_\cc$ is compatible with $\omega$.

Note that $\re{B}$ is a $(n,n)$ semi-Riemannian metric on $T\dual M$.
Since it agrees with $g$ on horizontal vectors, we denote it by $g$ as well.
Clearly, $J$ and $g$ are skew-compatible.
We also set $w=-\im{B}$. 
More explicitly, $w(\xi,\varepsilon)=g(J\xi,\varepsilon)$.
$w$ is also a $(n,n)$ semi-Riemannian metric.

In the above example, the symplectic form we obtained is not necessarily closed.
Our next example is a similar construction but the symplectic form we obtain at the end is the canonical symplectic form on $T\dual M$.
We obtain this result at the cost of losing some control of the almost complex structure.

Again, we start with a (real) 7-dimensional $G_2$ manifold $(M,\varphi)$ and we think of $g$ as an isomorphism between $TM$ and $T\dual M$.
Using this isomorphism, we think of $\varphi$ as an element of $\Gamma(\Lambda^3 TM)$.
Therefore, $(T\dual_p M,\varphi)$ is a $G_2$-space.
The vertical subspace of $T_\alpha T\dual M$ is canonically defined and isomorphic to $T\dual_{\pi(\alpha)}M$.
The vertical subbundle defines a Lagrangian 7-plane distribution on $(T\dual M,\omega_\mt{can})$.
The space of compatible almost complex structures on $(T\dual_\alpha TM,\Lambda=T\dual_{\pi(\alpha)}M,\varphi,\omega_{\mt{can}})$ is contractible by Lemma \ref{lem:technical}.
Therefore, one can find a global almost complex structure $J$ such that the complexification of $(\Lambda,\varphi)$ with respect to $J$ gives us a compatible triple $(\omega_{\mt{can}}, \varphi_\cc, g)$.
Compatibility here means compatibility at every point in the sense of subsection \ref{ssec:compatstr}.
\vspace{-0.1in}

\section{Deforming associative submanifolds in complexification}
Note that an associative submanifold $Y$ of a $G_2$ manifold $M$, naturally sits as both an isotropic associative submanifold and a $B$-real associative submanifold in the zero section of $T\dual M$.
We consider the infinitesimal deformations of $Y$ in which $Y$ stays isotropic associative in subsection \ref{ssec:defia} and $B$-real associative in subsection \ref{ssec:defra}.
We obtain Seiberg-Witten type equations from the former.

\newpage

\subsection{Deformation as isotropic associative}
\label{ssec:defia}
We denote the normal bundle of $Y$ in $M$ (resp. $T\dual M$) by $\nu_\rr Y$ (resp. $\nu_\cc Y$) and set $\vv=\nu_\rr Y \oplus J\nu_\rr Y$.
Then we have the following decomposition
\begin{equation}
\nu_\cc Y = JTY \oplus \vv.
\label{eq:normdecompY}
\end{equation}

Let $\sigma_t:Y\to T\dual M$ be a one parameter family of embeddings.
Without loss of generality, we may assume that $\dot \sigma_0$ is a section of $\Gamma(\nu_\cc Y)$.
Let $f\in\Gamma(JTY)$, $v\in\Gamma(\vv)$ with $\eta:=f+v = \dot \sigma_0$.
Also, let $\wt G := Gr(3,TT\dual M)\to T\dual M$ denote the Grassmann 3-plane bundle over $T\dual M$.
We can lift the embedding $Y\hookrightarrow T\dual M$ to $Y \hookrightarrow \wt G$ using the Gauss map.
Then, the infinitesimal deformation of $Y$ by $\eta$ induces and infinitesimal deformation of the lift as in \cite{AS08}.

For a tangent space $L=T_xY=\langle e_1,e_2,e_3\rangle$, infinitesimal deformation is given by
\begin{equation*}
  \dot L = \sum_{i=1}^3 e^i\otimes \mc L_{\eta}(e_i) \in T_L\wt G.
\end{equation*}
So, the conditions for $Y$ to stay isotropic associative are given by 
\begin{enumerate}
\item $\sum e_i\times\mc L_v(e_i)=0$
\item $\omega(e_i,\mc L_f(e_j))=\omega(e_j, \mc L_f(e_i))$
\end{enumerate}
by Lemma \ref{lem:tangentofia}.

Using the Levi-Civita connection $\nabla$ of $(T\dual M,g)$, we define a Dirac type operator
\begin{align}
\nonumber
\dirac_{A_0} &: \Omega^0(\nu_\cc Y) \to \Omega^0(\nu_\cc Y) \\
\dirac_{A_0}(v) &= \sum e_i\times \nabla_{e_i}(v).
\label{eq:diracop}
\end{align}
Note that in the role of Clifford multiplication we are using the cross product operation.

\begin{align*}
0 &= \sum e_i\times  \mc L_v(e_i) \\
&= \sum e_i\times  (\nabla_{v}e_i - \nabla_{e_i}v) \\
&= \sum e_i\times  \nabla_{v}e_i - \sum e_i\times \nabla_{e_i}v
\end{align*}

We set the perturbation parameter $a(v)=-\sum e_i\times \nabla_ve_i$.
So, we the last equation becomes
\begin{equation}
\dirac_{A}(v) = \dirac_{A_0}(v)+a(v) = 0
\label{eq:dirac1}
\end{equation}
where $A=A_0+a$.

For the isotropy condition, we choose a standard coordinate chart $(q^i,p^i)$ for the symplectic form so that $\omega = \sum dq^i\wedge dp^i$ where $(q^i)$ are coordinates on the base space and $(p^i)$ are fiber directions.
Write $\sigma_t^i=\sigma_t^i(x^1,x^2,x^3)=q^i(\sigma_t(x^1,x^2,x^3))$ for $1\le i \le 7$ and $\sigma_t^j=\sigma_t^j(x^1,x^2,x^3)=p^j(\sigma_t(x^1,x^2,x^3))$ for $8\le j \le 14$ where $(x^1,x^2,x^3)$ are local coordinates on $Y$.
Note that (possibly after reparametrization) we may assume that $(\sigma_t^1,\sigma_t^2,\sigma_t^3)= (x^1,x^2,x^3)$.
Furthermore, since the image of $\sigma_0$ lies in the 0-section of $T\dual M$, we may also assume $\sigma_0^j=0$ for $8\le j \le 14$.

During the deformation $Y$ stays isotropic if $\sigma_t\dual\omega = 0$.
Since
\begin{align*}
  \sigma_{t*}\frac{\partial}{\partial x^i} &=  \sum_{j=1}^7 \frac{\partial \sigma_t^j}{\partial x^i}\frac{\partial}{\partial q^j} + \frac{\partial \sigma_t^{j+7}}{\partial x^i}\frac{\partial}{\partial p^{j}} \\
  &= \sum_{j=1}^3 \delta^j_i\frac{\partial}{\partial q^j} + \sum_{j=4}^7 \frac{\partial \sigma_t^j}{\partial x^i}\frac{\partial}{\partial q^j} + \sum_{j=1}^7 \frac{\partial \sigma_t^{j+7}}{\partial x^i}\frac{\partial}{\partial p^{j}} \\
  &= \frac{\partial}{\partial q^i} + \sum_{j=4}^7 \frac{\partial \sigma_t^j}{\partial x^i}\frac{\partial}{\partial q^j} + \sum_{j=1}^7 \frac{\partial \sigma_t^{j+7}}{\partial x^i}\frac{\partial}{\partial p^{j}},
\end{align*}
we have
\begin{align}
  \nonumber
  0 &= \omega(\sigma_{t*}(\frac{\partial}{\partial x^i}),\sigma_{t*}(\frac{\partial}{\partial x^j})) \\
  &= \frac{\partial \sigma_t^{i+7}}{\partial x^j} - \frac{\partial \sigma_t^{j+7}}{\partial x^i}
  + \sum_{k=4}^7 \frac{\partial \sigma^{k}_t}{\partial x^i}\frac{\partial \sigma^{k+7}_t}{\partial x^j} - \frac{\partial \sigma^{k}_t}{\partial x^j}\frac{\partial \sigma^{k+7}_t}{\partial x^i}.
  \label{eq:globaldefisot}
\end{align}
Note that the last equation is already of the form $da = -q(\psi_1\otimes\psi_2)$ where $a$ is a 1-form on $Y$ given by
\begin{equation*}
  a = \sigma_t^8dx^1+\sigma_t^9dx^2+\sigma_t^{10}dx^3,
\end{equation*}
$\psi_1$ and $\psi_2$ are spinors living as sections of $\Omega^1(\nu_\rr Y)$ and $\Omega^1(J\nu_\rr Y)$ given by
\begin{align*}
  \psi_1 &= \sum_{i=1}^3 \frac{\partial}{\partial x^i}(\sigma^4_t, \dots, \sigma^7_t) dx^i, \\
  \psi_2 &= \sum_{j=1}^3 \frac{\partial}{\partial x^j}(\sigma^{11}_t, \dots, \sigma^{14}_t) dx^j,
\end{align*}
and $q$ is a bilinear map given by
\begin{equation*}
  q(\psi_1\otimes\psi_2) = \psi_1\times \psi_2
\end{equation*}
here the cross product is taken in the $1$-form parts with metric identification.


\subsection{Deformation as \texorpdfstring{$B$}{B}-real associative}
\label{ssec:defra}
We proceed as in subsection \ref{ssec:defia}.
The conditions for $Y$ to stay $B$-real associative are given as in Lemma \ref{lem:tangentofrealassocplanes}.
\begin{enumerate}
\item $\sum e_i\times\mc L_v(e_i)=0$
\item $g(Je_j,\mc L_f(e_i))+g(Je_i,\mc L_f(e_j))=0$
\end{enumerate}
The first condition is the same as before so we get the same equation
\begin{equation}
  \dirac_{A}(v) = \dirac_{A_0}(v)+a(v) = 0.
\end{equation}

For the second condition, we identify $f\in\Gamma(JTY)$ with $-Jf\in\Gamma(TY)$.
Then, we see that
\begin{align*}
  \mc L_f\left( g(e_i,e_j) \right) &=  \mc L_f(g)(e_i,e_j) + g(\mc L_fe_i,e_j)+g(e_i,\mc L_fe_j) \\
  \mc L_f\left( \delta_{ij}\right) &=  \mc L_f(g)(e_i,e_j).
\end{align*}
In other words, $f$ is the image of a Killing vector field on $Y$.

\vspace{.1in}

{\it Acknowledgements: We would like to thank Mahir Can for teaching us complex Lie groups, and the first named author would like to thank Sema Salur for collaborating \cite{AS08}, which eventually gave way to this work.}

\vspace{.05in}

\bibliographystyle{alpha}

\begin{thebibliography}{99999}
    
  \bibitem[AC15]{AC15} Akbulut, Selman, and Mahir Bilen Can. ``Complex $\text {G} _2 $ and associative grassmannian." arXiv preprint arXiv:1512.03191 (2015).
 
  \bibitem[AK16]{AK16} Akbulut, Selman, and Mustafa Kalafat. ``Algebraic topology of $\text{G}_2$ manifolds." Expositiones Mathematicae 34.1 (2016): 106-129.
 
  \bibitem[AS08]{AS08} Akbulut, Selman, and Sema Salur. ``Deformations in $\text{G}_2$ manifolds." Advances in Mathematics 217.5 (2008): 2130-2140.

  \bibitem[Bry87]{Bry87} Bryant, Robert L. ``Metrics with exceptional holonomy." Annals of mathematics (1987): 525-576.
  
    \bibitem[Bry05]{Bry05} Bryant, Robert L. ``Some remarks on $\text{G}_2$  structures." PGGT (2005): 75-109.
  
  \bibitem[HL82]{HL82} Harvey, Reese, and H. Blaine Lawson. ``Calibrated geometries." Acta Mathematica 148.1 (1982): 47-157.
  
  \bibitem[M98]{M98} McLean, Robert C. ``Deformations of calibrated submanifolds'', Comm. Anal. Geom.  6 (1998): 705--747.

  \bibitem[MS17]{MS17} McDuff, Dusa, and Dietmar Salamon. Introduction to symplectic topology. Oxford University Press, 2017.
  
  \bibitem[RS07]{RS07} Robles, Colleen and Sema Salur. ``Calibrated associative and Cayley embeddings.'' arXiv preprint arXiv:0708.1286 (2007).

  
  \bibitem[SV13]{SV13} Springer, Tonny A., and Ferdinand D. Veldkamp. Octonions, Jordan algebras and exceptional groups. Springer, 2013.
  
  
  \bibitem[T67]{T67} Thomas, Emery ``Postnikov invariants and higher order cohomology operations'', Ann. of Math. vol 85 (1967): 184--217.
  
  \bibitem[Y\i l17]{Yil17} Y\i ld\i r\i m, \"Ust\"un. "On the Minimal Compactification of the Cayley Grassmannian." arXiv preprint arXiv:1711.05169 (2017).
\end{thebibliography}

\end{document}